\documentclass[11pt]{article}

\usepackage[margin=1in]{geometry}
\usepackage{amsmath,amssymb}
\usepackage{booktabs}
\usepackage{longtable}
\usepackage{graphicx}
\usepackage{float}
\usepackage{hyperref}
\usepackage{enumitem}
\usepackage{algorithm}
\usepackage{algpseudocode}
\usepackage{booktabs}
\usepackage{tabularx}
\usepackage{multirow}
\usepackage{array}
\usepackage{tikz}
\usepackage{amsthm}
\newtheorem{definition}{Definition}
\newtheorem{theorem}{Theorem}

\newtheorem{proposition}{Proposition}[section]

\usetikzlibrary{arrows.meta,positioning}
\usepackage{hyperref}

\newcommand{\blue}[1]{\begin{color}{blue}#1\end{color}}

\title{GPU-Accelerated Preprocessing for Large-Scale Linear Programming}
\date{October 13, 2026}

\author{
Kaihuang Chen\textsuperscript{1}
\and
Jiawei Lu\textsuperscript{1}
\and
Defeng Sun\textsuperscript{1}
\and
Yancheng Yuan\textsuperscript{1}
\and
Guojun Zhang\textsuperscript{1}
\and
Xiangyu Zhang\textsuperscript{1}
\and
Xinyuan Zhao\textsuperscript{2}
}

\date{
\small
\textsuperscript{1}Department of Applied Mathematics,
The Hong Kong Polytechnic University, Hung Hom, Hong Kong.\\
\texttt{\{kaihuang.chen, jiawei.lu, guojun.zhang,
xiangyu.zhang\}@connect.polyu.hk}\\
\texttt{\{defeng.sun, yancheng.yuan\}@polyu.edu.hk}
\\[0.5em]
\textsuperscript{2}Department of Mathematics,
Beijing University of Technology, Beijing, P.R. China.\\
\texttt{xyzhao@bjut.edu.cn}
}

\begin{document}

\maketitle

\begin{abstract}
GPU-accelerated first-order methods are increasingly effective for solving
large-scale linear programs (LPs), yet LP preprocessing remains predominantly
CPU-based and can become costly.
A key challenge is that dependencies among presolve reductions
and frequent modifications to sparse LP data make efficient
GPU parallelization difficult.
To address this challenge, we develop GPU-Preprocessor, a GPU-native
LP preprocessing framework that integrates folding and presolve.
To efficiently implement folding on GPUs, we adapt a randomized partition refinement algorithm to LP folding, using sparse matrix–vector products and parallel sorting to construct stable row and column partitions.
For presolve, we introduce a new piecewise-linear reformulation rule that exploits underlying convex piecewise-linear structure to obtain sparser LP representations. We also develop a framework that combines phase scheduling with a Scan--Plan--Apply design to coordinate presolve rules and reduce sparse-matrix reconstruction overhead through batched updates.
Numerical experiments on three LP benchmark sets show that
GPU-Preprocessor achieves greater average problem-size reductions
than PSLP and Gurobi Presolve, with SGM10 preprocessing-time
speedups of up to $43.9\times$ and $677.7\times$, respectively.
Here, SGM10 denotes the shifted geometric mean with a 10-second shift.
On the largest tested quadratic assignment problem LP relaxation,
these speedups reach $126.3\times$ and $2{,}207.9\times$.
When integrated with the GPU-accelerated LP solver HPR-LP,
GPU-Preprocessor achieves SGM10 speedups of
$2.5\times$--$5.0\times$ in combined preprocessing and solving time
over HPR-LP without preprocessing.
The code is available at
\url{https://github.com/PolyU-IOR/GPU-Preprocessor}.
\end{abstract}



\blue{}

\section{Introduction}

This paper studies preprocessing for large-scale linear programs (LPs) on GPUs. We consider LPs of the form
\begin{equation}
\label{eq:original-lp}
\begin{aligned}
\min_{x\in\mathbb{R}^{n}}\quad
& c^\top x\\
\mathrm{s.t.}\quad
& Ax\in\mathcal{R},\\
& x\in\mathcal{C},
\end{aligned}
\end{equation}
where $A\in\mathbb{R}^{m\times n}$ is the constraint matrix,
$c\in\mathbb{R}^n$ is the objective vector,
\[
\mathcal{R}:=\{r\in\mathbb{R}^m:L\leq r\leq U\},
\qquad
\mathcal{C}:=\{x\in\mathbb{R}^n:\ell\leq x\leq u\},
\]
with
$L\in(\mathbb{R}\cup\{-\infty\})^m$,
$U\in(\mathbb{R}\cup\{+\infty\})^m$,
$\ell\in(\mathbb{R}\cup\{-\infty\})^n$, and
$u\in(\mathbb{R}\cup\{+\infty\})^n$.

Traditional LP solvers primarily rely on simplex or interior-point
methods~\cite{Gurobi2025,IBMCPLEX}. Before applying these algorithms,
presolve is typically used to simplify the LP by removing redundant
constraints, eliminating variables, and tightening
bounds~\cite{brearley1975analysis,andersen1995presolving}. These
reductions can reduce solution time and memory requirements and, in
some cases, detect infeasibility or unboundedness.

{Extensive research has developed a broad range of techniques
for LP preprocessing and solution recovery.}
Early work established systematic reduction procedures
and methods for exploiting structural
redundancy~\cite{brearley1975analysis,williams1983reduction,
bradley1983structural,tomlin1986duplicate},
while subsequent work developed comprehensive presolve
and solution-recovery frameworks and specialized
preprocessing techniques for simplex and interior-point
methods~\cite{tomlin1983formal,andersen1995presolving,
gondzio1997presolve,meszaros2003advanced}.
{Beyond conventional presolve reductions,
Grohe et al.~\cite{Grohe2014DimensionReduction}
developed a dimension-reduction method called folding,
which uses color refinement to construct stable partitions
for merging variables and constraints.}
More recently, Gleixner et al.~\cite{gleixner2023papilo}
developed PaPILO, a solver-independent presolving
library for LP and mixed-integer programming (MIP) that exploits multicore CPUs
through parallel reduction detection and conflict checking.

Meanwhile, GPU-accelerated first-order methods have attracted increasing attention for solving large-scale LPs~\cite{ApplegateEtAl2021,LuYang2025,LuEtAl2023,
ChenEtAl2026,LuPengYang2025}. Their iterations rely mainly on sparse
matrix--vector products, projections, and vector operations, which are
well suited to GPU parallelism. In particular, Applegate et
al.~\cite{ApplegateEtAl2021} developed PDLP, a practical LP solver based
on the primal--dual hybrid gradient method. GPU implementations of PDLP,
including cuPDLP.jl~\cite{LuYang2025} and
cuPDLP-C~\cite{LuEtAl2023}, have demonstrated competitive performance
relative to simplex and interior-point methods on large-scale LP
benchmarks. Chen et al.~\cite{ChenEtAl2026} developed HPR-LP, a
GPU-accelerated LP solver based on a Halpern Peaceman--Rachford method,
which further improves performance over cuPDLP. As GPU acceleration
reduces the time spent in the core solver, CPU-based presolve can account
for an increasing fraction of the overall solution time, particularly
for large-scale LPs. Consequently, reducing presolve cost becomes
increasingly important for realizing the full benefit of GPU-accelerated
LP solving.

Motivated by this setting, Cederberg and Boyd~\cite{CederbergBoyd2026}
developed PSLP, an open-source CPU presolver for GPU-accelerated
first-order LP solvers. PSLP uses a selected set of inexpensive rules
to balance reduction effectiveness and presolve time, achieving model
reductions comparable to those of Gurobi Presolve at lower presolve cost.
However, CPU-based presolve can remain costly relative to GPU solve
time, particularly for larger LPs. This motivates developing GPU-native preprocessing
techniques that reduce preprocessing overhead while
keeping LP data on the device throughout preprocessing
and solving.

Some parallelizable presolve operations have already been explored on
GPUs. Sofranac et al.~\cite{sofranac2022gpu} developed a GPU algorithm
for domain propagation that tightens variable bounds using linear
constraints. Nevertheless, publicly available general-purpose GPU-based
LP presolvers remain limited. A key challenge is that classical presolve
involves dependencies among reductions and frequent modifications to
sparse LP data. Reductions may enable, invalidate, or conflict with one
another, while fine-grained modifications to sparse data can incur
substantial data-management and synchronization overhead on GPUs. These
characteristics make efficient GPU parallelization challenging. To
address them, we develop GPU-Preprocessor, a GPU-native preprocessor for
large-scale LPs. \footnote{%
Our initial GPU presolve implementation adapted selected
implementation techniques from the CPU-based PSLP
solver~\cite{CederbergBoyd2026}
and was integrated into the Julia implementation of HPR-LP in
\href{https://github.com/PolyU-IOR/HPR-LP/commit/0ecf5f0b440dacb3bfe9a71c6709bbc43c567882}
{April 2026}. GPU folding was added to a development branch in
\href{https://github.com/PolyU-IOR/HPR-LP/commit/e010452eea08cd8d44c98104c846c53066293b54}
{June 2026},
and both components were subsequently integrated into HPR-LP-C in
\href{https://github.com/PolyU-IOR/HPR-LP-C/commit/ff3be826f89950a870902836ce21685da094b0d8}
{August 2026}.
Separately, Cederberg and Boyd~\cite{CederbergBoyd2026cuPSLP}
presented cuPSLP, a GPU-accelerated version of PSLP,
in a September 2026 preprint.
} Our main contributions are as follows.

\begin{enumerate}[leftmargin=*]

\item We develop GPU-Preprocessor, a GPU-native framework
that integrates folding and presolve for large-scale LPs.
To efficiently implement folding on GPUs, we adapt a
randomized partition refinement algorithm to LP folding,
using sparse matrix--vector products and parallel sorting
to construct stable row and column partitions.

\item  For presolve, we introduce a new piecewise-linear
reformulation rule that identifies the underlying convex
piecewise-linear structure of an LP, applies presolve
to the resulting problem, and selects the LP representation
with fewer nonzeros. To efficiently implement presolve on GPUs, we develop
a GPU-accelerated framework that combines phase scheduling with a
Scan--Plan--Apply design to coordinate multiple presolve
rules and reduce sparse-matrix reconstruction overhead
through batched updates.

\item Numerical experiments on three LP benchmark sets show that
GPU-Preprocessor achieves greater average reductions in problem size
than PSLP and Gurobi Presolve while requiring substantially less
preprocessing time.
Specifically, GPU-Preprocessor achieves SGM10 preprocessing-time
speedups of $6.0\times$--$43.9\times$ and
$23.2\times$--$677.7\times$ over PSLP and
Gurobi Presolve, respectively.
For the largest tested LP relaxation of a quadratic
assignment problem (QAP), these speedups reach
$126.3\times$ and $2{,}207.9\times$, respectively.
When combined with HPR-LP, GPU-Preprocessor further
achieves SGM10 total-time speedups of
$2.5\times$--$5.0\times$,
$1.5\times$--$3.0\times$, and
$1.9\times$--$29.0\times$
over HPR-LP without preprocessing, with PSLP,
and with Gurobi Presolve, respectively.

\end{enumerate}

The remainder of this paper is organized as follows.
Section~2 introduces the preliminaries of LP preprocessing
and reviews classical results.
Section~3 presents the GPU-Preprocessor framework,
which integrates GPU-accelerated folding and presolve.
Section~4 reports numerical results on several LP benchmark sets.
Section~5 concludes the paper.

\section{Preliminaries}
\label{sec:preliminaries}
This section recalls basic concepts and classical results on LP preprocessing
used in the subsequent development. We first define a valid preprocessing
operator and review stable partitions and folding, and then summarize
classical presolve reductions relevant to GPU-Preprocessor.

Let $P$ denote the original LP~\eqref{eq:original-lp}, with optimal value $v(P)$ and set of optimal solutions $\mathcal X^\star(P)$.

\begin{definition}[Valid preprocessing operator]
\label{def:valid-preprocessing-operator}
A preprocessing operator \(\mathcal T\) with \(\mathcal T(P)=P'\) is valid if
there exists a recovery map \(\mathcal Q\) such that, whenever \(P\) has an
optimal solution, \(P'\) also has an optimal solution and
\begin{equation*}
v(P')=v(P),
\qquad
\mathcal Q\bigl(\mathcal X^\star(P')\bigr)
\subseteq \mathcal X^\star(P).
\end{equation*}
\end{definition}

\subsection{Stable Partitions and Folding}
\label{sec:folding-preliminaries}
Folding reduces the dimension of an LP by aggregating rows and columns
according to stable partitions~\cite{Grohe2014DimensionReduction}.
We first introduce the necessary partition notation and the notion of
stability, then define the folded LP and recall its validity and solution
recovery.

For a nonnegative integer $k$, we write $[k]=\{1,\ldots,k\}$,
with $[0]=\varnothing$.
A partition $\mathcal P=\{S_1,\ldots,S_q\}$ of $[k]$
consists of pairwise disjoint nonempty subsets whose union is $[k]$.
The sets $S_1,\ldots,S_q$ are called the parts of $\mathcal P$, and
$|\mathcal P|$ denotes the number of parts. For each $s\in[k]$, let
$\pi_{\mathcal P}(s)$ denote the index of the part containing $s$, that is,
$\pi_{\mathcal P}(s)=i$ if $s\in S_i$. For two partitions $\mathcal P$
and $\mathcal P'$ of $[k]$, we say that $\mathcal P'$ refines $\mathcal P$
(or is finer than $\mathcal P$) if every part of $\mathcal P'$ is contained
in a part of $\mathcal P$; equivalently, $\mathcal P$ is coarser than
$\mathcal P'$.

For the original LP~\eqref{eq:original-lp}, let $\mathcal P_r=\{I_1,\ldots,I_{\bar m}\}, \mathcal P_c=\{J_1,\ldots,J_{\bar n}\}$
be partitions of the row and column index sets, respectively, where
$\bar m=|\mathcal P_r|$ and $\bar n=|\mathcal P_c|$.
Following~\cite{Grohe2014DimensionReduction}, folding requires these
partitions to be stable in the following sense.

\begin{definition}[Stable partition pair]
\label{def:folding-stable-revised-v2}
A partition pair $(\mathcal P_r,\mathcal P_c)$ is stable for
LP~\eqref{eq:original-lp} if, for every $I_p\in\mathcal P_r$ and
$J_q\in\mathcal P_c$, the following conditions hold:
\begin{enumerate}[label=\textup{(\roman*)},leftmargin=*]
\item the constraint bounds $(L_i,U_i)$ are identical within $I_p$,
and the objective coefficient and variable bounds $(c_j,\ell_j,u_j)$
are identical within $J_q$;
\item
\[
\sum_{j\in J_q}A_{ij}
=
\sum_{j\in J_q}A_{i'j}
\quad\text{for all } i,i'\in I_p,
\qquad
\sum_{i\in I_p}A_{ij}
=
\sum_{i\in I_p}A_{ij'}
\quad\text{for all } j,j'\in J_q.
\]
\end{enumerate}
\end{definition}

A stable partition pair for the original LP is coarsest if every
stable partition pair for this LP refines it in both the row and
column partitions. The coarsest stable partition pair exists and
is unique~\cite{Grohe2014DimensionReduction}; we denote it by
$(\mathcal P_r^\star,\mathcal P_c^\star)$.

Given a stable partition pair $(\mathcal P_r,\mathcal P_c)$,
folding~\cite{Grohe2014DimensionReduction} aggregates each row part into
one constraint and each column part into one variable, yielding
\begin{equation}
\label{eq:folded-lp}
\begin{aligned}
\min_{\bar x\in\mathbb{R}^{\bar n}}\quad
& \bar c^{\top}\bar x\\
\mathrm{s.t.}\quad
& \bar L\leq\bar A\bar x\leq\bar U,\\
& \bar\ell\leq\bar x\leq\bar u,
\end{aligned}
\end{equation}
where, for $p=1,\ldots,\bar m$ and $q=1,\ldots,\bar n$,
\[
\bar A_{pq}
=\frac{1}{|I_p|}\sum_{i\in I_p}\sum_{j\in J_q}A_{ij},
\qquad
\bar c_q=\sum_{j\in J_q}c_j.
\]
The folded constraint bounds $(\bar L_p,\bar U_p)$ and variable bounds
$(\bar\ell_q,\bar u_q)$ are the common bounds within $I_p$ and $J_q$,
respectively.

\begin{theorem}[Validity of folding~\cite{Grohe2014DimensionReduction}]
\label{thm:folding-validity}
For a stable partition pair $(\mathcal P_r,\mathcal P_c)$, folding is a
valid preprocessing operator in the sense of
Definition~\ref{def:valid-preprocessing-operator}. In particular, any
optimal solution $\bar x$ of~\eqref{eq:folded-lp} can be
lifted to an optimal solution $x$ of~\eqref{eq:original-lp} by
\begin{equation}
\label{eq:folding-lift}
x_j=\bar x_{\pi_{\mathcal P_c}(j)},
\qquad j=1,\ldots,n.
\end{equation}
\end{theorem}

\subsection{Classical Presolve}
\label{sec:presolve-preliminaries}

Presolve simplifies an LP by applying valid reductions to its constraints,
variables, and bounds before optimization. Table~\ref{tab:presolve-strategies}
summarizes the presolve rules implemented in GPU-Preprocessor, including
classical rules and the new rule introduced in this paper.

\begingroup
\footnotesize
\setlength{\tabcolsep}{6pt}
\renewcommand{\arraystretch}{1.10}

\begin{longtable}{
    @{}
    >{\raggedright\arraybackslash}p{0.27\linewidth}
    >{\raggedright\arraybackslash}p{0.68\linewidth}
    @{}
}
\caption{Presolve rules implemented in GPU-Preprocessor.}
\label{tab:presolve-strategies}\\

\toprule
Presolve rule & Description \\
\midrule
\endfirsthead

\multicolumn{2}{c}{
    \tablename~\thetable{} (continued)
}\\
\toprule
Presolve rule & Description \\
\midrule
\endhead

\midrule
\multicolumn{2}{r}{Continued on next page}\\
\endfoot

\bottomrule
\endlastfoot

\multicolumn{2}{@{}l}{\textit{Existing presolve rules}}\\*
\addlinespace

Infeasible and fixed variables~\cite{andersen1995presolving}
&
Detects infeasibility when a variable's lower bound exceeds its upper
bound. Fixes a variable when its bounds coincide and substitutes its
value into the constraints and objective.
\\

\addlinespace

Infeasible and redundant rows~\cite{brearley1975analysis,andersen1995presolving}
&
Uses variable bounds to detect infeasible rows and remove redundant rows.
\\

\addlinespace

Empty rows and columns~\cite{andersen1995presolving}
&
Removes redundant empty rows and fixes variables corresponding to empty
columns according to their objective coefficients and bounds.
\\

\addlinespace

Singleton rows~\cite{brearley1975analysis}
&
Uses a row with one nonzero entry to tighten the bounds of its variable
and then removes the row.
\\

\addlinespace

Bound strengthening~\cite{brearley1975analysis,AchterbergEtAl2019}
&
Tightens variable bounds using constraint information and the bounds of
the other variables.
\\

\addlinespace

Singleton columns~\cite{brearley1975analysis,andersen1995presolving}
&
Substitutes or fixes a variable appearing in only one row, or tightens
its bounds, depending on the row type, variable bounds, and objective
coefficient.
\\

\addlinespace

Dual fixing~\cite{AchterbergEtAl2019}
&
Uses objective and constraint information to fix a variable at one of its
bounds or detect unboundedness.
\\

\addlinespace

Doubleton equations~\cite{bradley1983structural}
&
Uses an equality row with two nonzero entries to eliminate one variable
by substitution.
\\

\addlinespace

Parallel rows~\cite{tomlin1986duplicate,AchterbergEtAl2019}
&
Detects rows whose coefficients differ only by a nonzero scalar factor,
combines their bounds after scaling, and removes redundant rows.
\\

\addlinespace

Parallel columns~\cite{andersen1995presolving,AchterbergEtAl2019}
&
Detects columns whose coefficients differ only by a nonzero scalar factor
and combines or fixes the corresponding variables according to their
objective coefficients and bounds.
\\

\addlinespace

Column domination~\cite{beasley1987algorithm}
&
Fixes cost-dominated variables to zero in set covering
LP relaxations.
\\

\addlinespace

Zero-cost product variables~\cite{mccormick1976computability,GamrathEtAl2016}
&
Eliminates a zero-cost auxiliary product variable and its
associated McCormick constraints when the variable appears
in no other constraints and has no additional nonredundant bounds.
\\

\midrule
\multicolumn{2}{@{}l}{
    \textit{Presolve rule introduced in this paper}
}\\*
\addlinespace

Piecewise-linear reformulation
&
Recognizes a convex piecewise-linear program from its LP
representation, presolves the recognized program, and selects
the LP representation with fewer matrix nonzeros.
\\

\end{longtable}
\endgroup

\section{GPU-Preprocessor}

\subsection{Overview of GPU-Preprocessor}
\label{sec:gpu-presolve-framework}

GPU-Preprocessor consists of two components: folding and presolve.
Folding reduces LP dimensions by merging rows and columns,
as described in Section~\ref{sec:gpu-accelerated-folding}.
Presolve organizes the rules in Table~\ref{tab:presolve-strategies}
into five phases: Startup, Cleanup, Propagation, Aggregation,
and Parallel Detection, as described in
Section~\ref{sec:gpu-accelerated-presolve}.
Algorithm~\ref{alg:gpu-preprocessor} summarizes the overall workflow.
Cleanup is repeated until no further changes occur,
and subsequent phases return to Cleanup when sufficient
progress is made.
We limit the total number of phase executions to $K$,
excluding Cleanup because of its low computational cost.
Here, nnz denotes the number of nonzero entries
in the constraint matrix of the current LP.
After the preprocessed LP is solved, postsolve applies the
recorded presolve recovery maps in reverse order.
Unfolding then applies the folding recovery map to recover
a solution of the original LP.

\begin{algorithm}[H]
\caption{GPU-Preprocessor workflow}
\label{alg:gpu-preprocessor}
\begin{algorithmic}[1]
\Require Original LP $P$ and phase limit $K$
\Ensure Recovered optimal solution $x$ of the original LP $P$

\State $P^0 \gets \Call{Folding}{P}$
\State $P^1 \gets \Call{Startup}{P^0}$; $r \gets 1$

\While{true}
    \Statex \textbf{Cleanup:}
    \Repeat
        \State $P^r \gets \Call{Cleanup}{P^r}$
    \Until{Cleanup leaves the LP unchanged}

    \State \textbf{break if} $r \ge K$

    \State $P^{r+1} \gets \Call{Propagation}{P^r}$; $r \gets r+1$

    \If{Propagation reduces nnz sufficiently or tightens variable bounds}
        \State \textbf{go to Cleanup}
    \EndIf

    \State \textbf{break if} $r \ge K$

    \State $P^{r+1} \gets \Call{Aggregation}{P^r}$; $r \gets r+1$
    \If{Aggregation reduces nnz sufficiently}
        \State \textbf{go to Cleanup}
    \EndIf

    \State \textbf{break if} $r \ge K$

    \State $P^{r+1} \gets \Call{Parallel Detection}{P^r}$; $r \gets r+1$
    \If{Parallel Detection reduces nnz sufficiently}
        \State \textbf{go to Cleanup}
    \EndIf
    \State \textbf{break}
\EndWhile

\State $x^r \gets \Call{Solve}{P^r}$
\State $x^0 \gets \Call{Postsolve}{x^r}$
\State $x \gets \Call{Unfold}{x^0}$
\State \Return $x$
\end{algorithmic}
\end{algorithm}

The following theorem establishes the validity of
Algorithm~\ref{alg:gpu-preprocessor} and the associated
solution recovery.

\begin{theorem}
\label{thm:preprocessing-postprocessing-correctness}
Suppose that the original LP $P$ has an optimal solution,
and let $P^r$ be the final preprocessed LP passed to Solve
in Algorithm~\ref{alg:gpu-preprocessor}.
Then, with probability one, the preprocessed LP $P^r$ has
the same optimal value as $P$, and the recovered solution
$x$ is optimal for $P$ whenever Solve returns an optimal
solution $x^r$ of $P^r$.

\end{theorem}

\begin{proof}
By Theorem~\ref{thm:folding-refinement-partition}, the partition
pair used for folding is stable with probability one.
On this event, folding is valid by
Theorem~\ref{thm:folding-validity}.
The piecewise-linear reformulation rule is valid by
Theorem~\ref{thm:unified-presolve-validity}, and the remaining
presolve rules are valid preprocessing operators.
Applying Definition~\ref{def:valid-preprocessing-operator}
successively shows that $P^r$ and $P$ have the same optimal value.
If Solve returns an optimal solution $x^r$ of $P^r$,
postsolve and unfolding apply the corresponding recovery
maps in reverse order to recover an optimal solution $x$ of $P$.
\end{proof}

\subsection{GPU-Accelerated Folding}
\label{sec:gpu-accelerated-folding}

To construct a stable partition pair for folding, we refine
row and column partitions $\mathcal P_r,\mathcal P_c$ satisfying
condition~(i) of Definition~\ref{def:folding-stable-revised-v2}
until condition~(ii) also holds.
Classical methods refine one partition using a single part of
the other~\cite{Grohe2014DimensionReduction}.
This can require many steps with small workloads, limiting GPU
utilization.

To exploit GPU parallelism, we develop a randomized partition refinement
method for LPs, inspired by hashed color refinement for unweighted
graphs~\cite{Kersting2014PowerColorRefinement}.
Each refinement uses a single sparse matrix--vector multiplication (SpMV)
to combine coefficient sums over all parts of the opposite partition
into randomized signatures.
These signatures are then used to refine the current partitions.
Algorithm~\ref{alg:folding-gpu-refinement-partition} summarizes this method.
We use a threshold $\theta\in(0,1]$ to skip folding by returning
the trivial stable partition pair when neither the row count nor
the column count can be reduced substantially.

\begingroup
\setlength{\intextsep}{8pt plus 2pt minus 2pt}
\begin{algorithm}[!htbp]
\caption{Randomized partition refinement for problem~\eqref{eq:original-lp}}
\label{alg:folding-gpu-refinement-partition}
\begin{algorithmic}[1]
\Require Dimension threshold $\theta\in(0,1]$
\Ensure A stable partition pair $(\mathcal P_r,\mathcal P_c)$, with probability one
\State Call Algorithm~\ref{alg:induced-partition-gpu} with
$\bigl((c_j,\ell_j,u_j)\bigr)_{j\in[n]}$ to obtain $\mathcal P_c^0$ \Comment{\textit{Column initialization}}
\State Call Algorithm~\ref{alg:induced-partition-gpu} with
$\bigl((L_i,U_i)\bigr)_{i\in[m]}$ to obtain $\mathcal P_r^0$ \Comment{\textit{Row initialization}}
\State Initialize the refinement round counter $t$ to $0$
\Repeat
  \If{$|\mathcal P_r^t|>\theta m$ and $|\mathcal P_c^t|>\theta n$}
    \State \Return $(\{\{1\},\ldots,\{m\}\},\{\{1\},\ldots,\{n\}\})$
  \EndIf
  \State Sample $\xi\in\mathbb R^{|\mathcal P_r^t|}$ with independent standard normal entries
  \State Set $s\in\mathbb R^m$ with $s_i=\xi_{\pi_{\mathcal P_r^t}(i)}$ for $i\in[m]$, and compute $h_c=A^\top s$
  \State Call Algorithm~\ref{alg:induced-partition-gpu} with
  $\bigl((\pi_{\mathcal P_c^t}(j),(h_c)_j)\bigr)_{j\in[n]}$
  to obtain $\mathcal P_c^{t+1}$ \Comment{\textit{Column refinement}}
  \State Sample $\eta\in\mathbb R^{|\mathcal P_c^{t+1}|}$ with independent standard normal entries
  \State Set $v\in\mathbb R^n$ with $v_j=\eta_{\pi_{\mathcal P_c^{t+1}}(j)}$ for $j\in[n]$, and compute $h_r=Av$
  \State Call Algorithm~\ref{alg:induced-partition-gpu} with
  $\bigl((\pi_{\mathcal P_r^t}(i),(h_r)_i)\bigr)_{i\in[m]}$
  to obtain $\mathcal P_r^{t+1}$ \Comment{\textit{Row refinement}}
  \State Increment $t$ by $1$
\Until{$\mathcal P_r^t=\mathcal P_r^{t-1}$ and $\mathcal P_c^t=\mathcal P_c^{t-1}$}
\State \Return $(\mathcal P_r^t,\mathcal P_c^t)$
\end{algorithmic}
\end{algorithm}

Algorithm~\ref{alg:induced-partition-gpu} constructs the initial and refined
partitions used in Algorithm~\ref{alg:folding-gpu-refinement-partition}
by grouping indices whose corresponding entries in the input tuple are equal.
It uses sort-based grouping~\cite{HellersteinHaasWang1997}, which is well suited
to parallel execution on GPUs~\cite{HeEtAl2009,GurumurthyEtAl2023}.

\begin{algorithm}[!htbp]
\caption{Grouping indices with equal tuple entries}
\label{alg:induced-partition-gpu}
\begin{algorithmic}[1]
\Require A tuple $\lambda=(\lambda_i)_{i\in[k]}$, where $k$ is a positive integer
\Ensure \begin{tabular}[t]{@{}l@{}}
A partition $\mathcal P$ of $[k]$ such that
for all $i,j\in[k]$,
$\pi_{\mathcal P}(i)=\pi_{\mathcal P}(j)$ if and only if $\lambda_i=\lambda_j$
\end{tabular}
\State \begin{tabular}[t]{@{}l@{}}
Find a permutation $\sigma$ of $[k]$ such that\\
for all $i,j,i'\in[k]$ with $i\le j\le i'$,
if $\lambda_{\sigma(i)}=\lambda_{\sigma(i')}$, then $\lambda_{\sigma(j)}=\lambda_{\sigma(i)}$
\end{tabular}
\State \begin{tabular}[t]{@{}l@{}}
Set $b\in\{0,1\}^k$ with $b_1=1$ and,
for $i\in[k-1]$,
$b_{i+1}=1$ if and only if $\lambda_{\sigma(i+1)}\ne\lambda_{\sigma(i)}$
\end{tabular}
\State Compute $g\in\mathbb Z^k$ with $g_i=\sum_{j=1}^{i}b_j$ for each $i\in[k]$
\State \Return $\mathcal P=
\bigl\{\{\sigma(i)\mid i\in[k],\ g_i=j\}\mid j\in[g_k]\bigr\}$
\end{algorithmic}
\end{algorithm}
\endgroup

The following theorem provides theoretical guarantees for
Algorithm~\ref{alg:folding-gpu-refinement-partition}; its proof is given in
Appendix~\ref{app:folding-refinement-proof}.

\begin{theorem}
\label{thm:folding-refinement-partition}
For any $\theta\in(0,1]$, with probability one,
Algorithm~\ref{alg:folding-gpu-refinement-partition} terminates within
$2+\min\{m-|\mathcal P_r^0|,\,n-|\mathcal P_c^0|\}$ refinement rounds
and returns the unique coarsest stable partition pair
$(\mathcal P_r^\star,\mathcal P_c^\star)$ for the original LP if
$|\mathcal P_r^\star|\le\theta m$ or $|\mathcal P_c^\star|\le\theta n$,
and the trivial stable partition pair otherwise.
\end{theorem}

For the returned stable partition pair, the folded LP is defined
by~\eqref{eq:folded-lp}. After the folded LP is solved, unfolding
recovers an optimal solution of the original LP
using~\eqref{eq:folding-lift}.

\subsection{GPU-Accelerated Presolve}
\label{sec:gpu-accelerated-presolve}
This subsection presents the presolve component of GPU-Preprocessor.
Section~\ref{sec:presolve-rules} summarizes the presolve phases
and introduces the piecewise-linear reformulation rule.
Section~\ref{sec:gpu-rule-implementation} presents the
Scan--Plan--Apply framework for batched presolve and
GPU implementations of more complex presolve strategies.
\subsubsection{Presolve Rules in GPU-Preprocessor}
\label{sec:presolve-rules}

Presolve begins with Startup, which applies presolve rules
designed for particular problem classes to reduce the number
of nonzeros in the constraint matrix.
Following Startup, the remaining presolve rules are organized
into four phases: Cleanup, Propagation, Aggregation, and
Parallel Detection, as summarized in
Table~\ref{tab:presolve-rule-classes}.
Cleanup applies trivial rules to simplify the LP.
Propagation derives implied bounds by propagating existing
bounds along individual rows and columns.
Aggregation reduces the problem size through variable
substitution. Parallel Detection handles parallel rows and columns.

\begin{table}[H]
\centering
\caption{Presolve rules by scheduling phase.}
\label{tab:presolve-rule-classes}
\vspace{4pt}
\small
\setlength{\tabcolsep}{4pt}
\begin{tabularx}{\linewidth}{@{}l>{\raggedright\arraybackslash}X@{}}
\toprule
Phase & Rules \\
\midrule
\textbf{Startup} &
Column domination;
zero-cost product variables;
piecewise-linear reformulation. \\
\addlinespace[8pt]
\textbf{Cleanup} &
Infeasible and fixed variables;
empty rows and columns;
singleton rows;
singleton columns. \\
\addlinespace[8pt]
\textbf{Propagation} &
Bound strengthening;
infeasible and redundant rows;
dual fixing. \\
\addlinespace[8pt]
\textbf{Aggregation} &
Doubleton equations. \\
\addlinespace[8pt]
\textbf{Parallel Detection} &
Parallel rows;
parallel columns. \\
\bottomrule
\end{tabularx}
\end{table}

In addition to existing presolve rules, we propose a new
\emph{piecewise-linear reformulation rule}
for LP representations of convex piecewise-linear programs
(P-LPs),
such as
$\ell_1$ regression~\cite{Wagner1959,BarrodaleRoberts1973},
$\ell_\infty$ regression~\cite{Wagner1959,BarrodaleYoung1966},
quantile regression~\cite{KoenkerBassett1978}, and
basis pursuit~\cite{ChenDonohoSaunders1998}.
This rule exploits the fact that a convex P-LP admits LP representations
exhibiting either of the following two structural patterns:

\begin{enumerate}[
  label=\textit{Pattern~\arabic*.},
  leftmargin=*
]
\item The LP contains two constraints
$x_j\geq x_k$ and $x_j\geq-x_k$, where $j\ne k$.

\item The LP contains two variables $x_p,x_q$ with $p\ne q$
satisfying $\ell_p=\ell_q=0$, $c_p+c_q\geq0$, and, for every row $i$,
$A_{ip}+A_{iq}\geq0$ if $U_i<+\infty$ and
$A_{ip}+A_{iq}\leq0$ if $L_i>-\infty$.
\end{enumerate}

\noindent
The rule first detects Pattern~1 or~2 in the LP and constructs
the associated convex P-LP. It then applies valid presolve to this problem.
Finally, it compares two LP representations of the presolved convex P-LP,
one exhibiting each pattern, and selects the representation with fewer
nonzeros in the constraint matrix as the presolved LP.
Figure~\ref{fig:unified-presolve-flow} illustrates this procedure.
\begin{figure}[H]
\centering
\includegraphics[width=\linewidth]{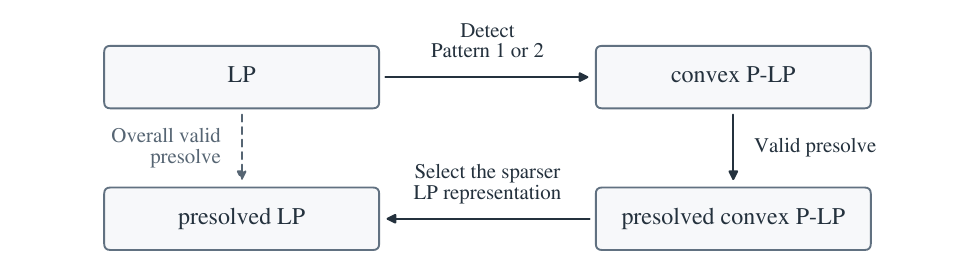}
\caption{Workflow of the piecewise-linear reformulation rule.}
\label{fig:unified-presolve-flow}
\end{figure}

For example, let $X\subseteq\mathbb R^n$ be a polyhedron, let
$A\in\mathbb R^{m\times n}$ and $b\in\mathbb R^m$, and let
$w\in\mathbb R^m$ have positive entries.
We consider presolve for LPs of the following form,
with variables $x\in\mathbb R^n$ and $y,z\in\mathbb R^m$:
\[
\min\{w^\top z\mid Ax+y=b,\ x\in X,\ z\geq y,\ z\geq-y\}.
\label{eq:unified-l1-original-lp}
\]
The LP exhibits Pattern~1 and represents the weighted
$\ell_1$ regression problem~\cite{BloomfieldSteiger1983,BarrodaleRoberts1978}
\[
\min_{x\in X}\
\|\operatorname{diag}(w)(b-Ax)\|_1.
\label{eq:unified-l1-underlying}
\]
For a nonempty index set $I\subseteq[m]$ satisfying
$A_{i,:}=A_{j,:}$ and $b_i=b_j$ for all $i,j\in I$,
choose any $k\in I$ and define $J=([m]\setminus I)\cup\{k\}$.
The presolved convex P-LP is
\begin{equation}
\min_{x\in X}\
\|\operatorname{diag}(\bar w_J)(b_J-A_{J,:}x)\|_1,
\label{eq:unified-l1-presolved-underlying}
\end{equation}
where $\bar w_J=(\bar w_j)_{j\in J}$ is defined by
$\bar w_k=\sum_{i\in I}w_i$ and
$\bar w_j=w_j$ for $j\in J\setminus\{k\}$.
Two LP representations of \eqref{eq:unified-l1-presolved-underlying},
exhibiting Patterns~1 and~2, respectively, are
\begin{align}
&\min\{\bar w_J^\top z_J\mid
A_{J,:}x+y_J=b_J,\ x\in X,\
z_J\geq y_J,\ z_J\geq-y_J\},
\label{eq:unified-l1-intermediate-lp}\\[3pt]
&\min\{\bar w_J^\top(y_J^++y_J^-)\mid
A_{J,:}x+y_J^+-y_J^-=b_J,\ x\in X,\
y_J^+,y_J^-\geq0\},
\label{eq:unified-l1-final-lp}
\end{align}
where $y_J,z_J,y_J^+,y_J^-\in\mathbb R^{|J|}$ have components
indexed by $J$.
The LP \eqref{eq:unified-l1-final-lp} has $3|J|$ fewer nonzero entries
in its constraint matrix than the LP \eqref{eq:unified-l1-intermediate-lp}
and is therefore selected as the presolved LP.
Postsolve keeps $x$ and recovers
\[
y_i=\begin{cases}
y_i^+-y_i^-, & i\in J,\\[2pt]
y_k^+-y_k^-, & i\in I\setminus\{k\},
\end{cases}
\qquad
z_i=\begin{cases}
y_i^++y_i^-, & i\in J,\\[2pt]
y_k^++y_k^-, & i\in I\setminus\{k\}.
\end{cases}
\]

The following theorem establishes the validity of this rule;
its proof is given in Appendix~\ref{app:unified-presolve-proof}.

\begin{theorem}
\label{thm:unified-presolve-validity}
The piecewise-linear reformulation rule is a valid preprocessing operator
in the sense of
Definition~\ref{def:valid-preprocessing-operator}.
\end{theorem}

\subsubsection{GPU Implementation of Presolve Rules}
\label{sec:gpu-rule-implementation}

Some presolve rules are difficult to parallelize efficiently
on GPUs. Variable substitutions in the Aggregation phase
can cause conflicting updates, while rules in the Parallel
Detection phase require repeated comparisons of
indices and values of nonzero entries.
The scheduling scheme in Algorithm~\ref{alg:gpu-preprocessor}
therefore first executes the relatively inexpensive Cleanup
and Propagation phases after Startup.
The Aggregation and Parallel Detection phases are invoked
when the preceding phases make limited progress, with returns
to Cleanup to exploit newly enabled reductions.

We illustrate our specialized GPU implementations with
the doubleton equation rule in the Aggregation phase
and the parallel row and column rules in the Parallel Detection phase.
For the doubleton equation rule, candidate equations are detected
in parallel. A subset of equations that do not share variables
is then selected for batched substitution on the GPU,
avoiding conflicting updates to variable bounds and objective
coefficients. Changes to row bounds are summed using atomic
operations when substitutions affect the same row.
A batched update then applies the substitutions to $A$
and removes the eliminated rows and columns.
The transpose $A^\top$ is reconstructed from the updated $A$. For the parallel row and column rules, we first identify
candidates using GPU SpMV operations with random vectors.
These candidates are then grouped by hashes of their
indices and normalized values of nonzero entries.
Proportionality checks are performed only within these
groups before the corresponding rules are applied.

Repeated matrix reconstruction across presolve rules is also
costly on GPUs. Many rules, however, only delete rows or columns
and update bounds or objective data without changing retained
matrix coefficients. We use row and column deletion masks
to mark rows and columns for removal. Using these masks,
our Scan--Plan--Apply framework
accumulates updates across rules and defers matrix reconstruction.
\textsc{Scan} collects row and column information in parallel,
using the deletion masks to skip entries in rows or columns
marked for removal. \textsc{Plan} evaluates presolve rules
using this information and records deletions and vector
updates in a shared plan $\Pi$.
\textsc{Apply} then commits the accumulated vector updates
and reconstructs $A$ and $A^\top$ once for the deletions
recorded across multiple rules.

\section{Numerical Experiments}

We compare GPU-Preprocessor with PSLP~\cite{CederbergBoyd2026}
and the presolver in Gurobi~\cite{Gurobi2025}.
Section~\ref{sec:experimental-setup} describes the benchmarks,
hardware and software settings, and evaluation metrics.
Section~\ref{sec:ablation} examines the contributions of
folding and presolve.
Section~\ref{sec:presolve-performance} compares preprocessing
time and reduction strength.
Section~\ref{sec:end-to-end} evaluates solve and total times
when each method is combined with HPR-LP.
Finally, Section~\ref{sec:gpu-advantage-at-scale} examines
how performance changes with problem size on a controlled
family of QAP LP relaxations.

\subsection{Experimental Setup}
\label{sec:experimental-setup}

\paragraph{Benchmark datasets.}
We use three benchmark groups to evaluate GPU-Preprocessor on
general-purpose, large-scale sparse, and highly structured LPs. We use
49 publicly available instances from Mittelmann's LP benchmark
set~\cite{Mittelmann1999}. Following the size
classification in HPR-LP~\cite{ChenEtAl2026}, we select the 18
root-node LP relaxations from MIPLIB 2017~\cite{GleixnerEtAl2021} whose
constraint matrices contain more than $10^7$ nonzeros; we refer to this
subset as MIPLIB Large. QAP20 consists of 20 large, highly structured
LPs generated from QAPLIB instances~\cite{BurkardEtAl1997}.

\paragraph{Hardware.}
All experiments run on the same server equipped with two
Intel Xeon Gold 6548Y+ CPUs, each with 32 physical cores,
and four NVIDIA H100 PCIe GPUs, each with 80 GB of memory.
GPU-Preprocessor and HPR-LP use one GPU per run, while
PSLP and Gurobi perform presolve on the CPUs of the same
server.

\paragraph{Setup.}
We compare GPU-Preprocessor against PSLP version 0.0.11
and Gurobi Presolve version 12.0.2 under an academic license.
Unless otherwise stated, GPU-Preprocessor enables both
folding and presolve, PSLP uses its full configuration,
and Gurobi uses automatic presolve
(\texttt{Presolve=-1}).
To evaluate the effect of preprocessing on solver performance,
we use HPR-LP-C 0.1.3, the C implementation of HPR-LP.
The HPR-LP baseline disables both folding and presolve.

Preprocessing time is measured as wall-clock time,
excluding MPS parsing and CPU--GPU transfers of LP data.
For GPU-Preprocessor, it includes all enabled folding
and presolve stages.
Solve time is the algorithm time reported by HPR-LP-C
on the preprocessed LP, or on the original LP when
preprocessing is disabled.
For each instance, total time is the sum of preprocessing
and solve times; it equals solve time when preprocessing
is disabled.
Postsolve and validation on the original LP are excluded
from these timing metrics.

\paragraph{Performance metrics.}
We use the following metrics.
\begin{itemize}[leftmargin=*]
    \item \textbf{Shifted geometric mean.} For times \(t_i\) measured in seconds over \(K\) instances,
we report
    \[
    \operatorname{SGM10}(t)
    =
    \left(\prod_{i=1}^{K}(t_i+10)\right)^{1/K}-10.
    \]
We use SGM10 for preprocessing, solve, and total times; the shift reduces
the influence of very small values. Total time is computed for
each instance before applying SGM10.

    \item \textbf{Solved instances.}
    A run is counted as solved if HPR-LP meets its relative
    Karush--Kuhn--Tucker (KKT) stopping criterion with tolerance
    \(\epsilon=10^{-6}\) within the solve-time limit.
    For preprocessed runs, the criterion is evaluated on the preprocessed LP. The solve-time limit is 1000 seconds and excludes preprocessing time.
    Runs that do not meet the stopping criterion within the limit are counted
    as unsolved, and their solve times are set to that limit when computing
    aggregate metrics.

    \item \textbf{Reduction ratios.} We define
    \[
    r_m=\frac{m_{\mathrm{red}}}{m},\qquad
    r_n=\frac{n_{\mathrm{red}}}{n},\qquad
    r_{\mathrm{nnz}}=\frac{\mathrm{nnz}_{\mathrm{red}}}{\mathrm{nnz}}.
    \]
    The reported ratios are arithmetic means over each benchmark group, and smaller values indicate stronger reductions.
\end{itemize}

\subsection{Ablation Studies}
\label{sec:ablation}

We evaluate the contributions of folding and presolve by
comparing three configurations of HPR-LP: no preprocessing,
folding only, and folding followed by presolve.
The comparisons cover all three benchmark groups and follow
the setup in Section~\ref{sec:experimental-setup}.

Table~\ref{tab:ablation-components} shows that folding alone
reduces the mean nonzero ratio and lowers total-time SGM10
by 3.6\%--42.9\% relative to no preprocessing across
the three benchmark groups.
Adding presolve further reduces the mean nonzero ratio
and lowers total-time SGM10 by 31.7\%--73.5\% relative
to folding alone, despite the additional preprocessing cost.
On Mittelmann, the solved count increases from 46 without
preprocessing to 47 with folding alone and 49 with
folding followed by presolve. These results highlight the importance of both folding
and presolve in reducing problem size and total solution time.

\begin{table}[H]
\centering
\caption{Ablation study of folding and presolve on the three
benchmark groups. Bold values indicate the lowest total-time
SGM10 in each group.}
\label{tab:ablation-components}
\vspace{4pt}
\small
\setlength{\tabcolsep}{4pt}
\resizebox{\linewidth}{!}{%
\begin{tabular}{llrrrrrrr}
\toprule
Dataset & Preprocessing & Solved & Preprocessing time (s)
& Solve time (s) & Total time (s)
& \(r_m\) & \(r_n\) & \(r_{\mathrm{nnz}}\) \\
\midrule
Mittelmann
& None
& 46 & 0.000 & 14.456 & 14.456 & 1.000 & 1.000 & 1.000 \\
& Folding only
& 47 & 0.083 & 13.831 & 13.934 & 0.768 & 0.747 & 0.746 \\
& Folding + presolve
& 49 & 0.184 & 5.615 & \textbf{5.797} & 0.583 & 0.580 & 0.572 \\
\midrule
MIPLIB Large
& None
& 17 & 0.000 & 27.322 & 27.322 & 1.000 & 1.000 & 1.000 \\
& Folding only
& 17 & 0.197 & 15.399 & 15.606 & 0.519 & 0.509 & 0.454 \\
& Folding + presolve
& 17 & 0.292 & 10.407 & \textbf{10.666} & 0.402 & 0.384 & 0.342 \\
\midrule
QAP20
& None
& 20 & 0.000 & 33.172 & 33.172 & 1.000 & 1.000 & 1.000 \\
& Folding only
& 20 & 0.147 & 24.995 & 25.207 & 0.744 & 0.744 & 0.744 \\
& Folding + presolve
& 20 & 0.271 & 6.407 & \textbf{6.687} & 0.014 & 0.372 & 0.301 \\
\bottomrule
\end{tabular}%
}
\end{table}

\subsection{Preprocessing Performance}
\label{sec:presolve-performance}

We compare GPU-Preprocessor, PSLP, and Gurobi Presolve
in terms of preprocessing time and reduction strength
on the three benchmark groups.
Table~\ref{tab:preprocessing-performance} shows that
GPU-Preprocessor achieves speedups of
\(6.0\times\)--\(43.9\times\) over PSLP and
\(23.2\times\)--\(677.7\times\) over Gurobi Presolve
in terms of preprocessing-time SGM10.
Together with its lower mean row, column, and nonzero
ratios on all three groups, these speedups show that
GPU-Preprocessor improves both preprocessing efficiency
and reduction strength relative to the baselines.

\begin{table}[H]
\centering
\caption{Preprocessing time and reduction strength on the
three benchmark groups. Bold values indicate the lowest
value for each metric in each group.}
\label{tab:preprocessing-performance}
\vspace{4pt}
\small
\setlength{\tabcolsep}{4pt}
\begin{tabular}{llrrrr}
\toprule
Dataset & Preprocessor & Preprocessing time (s)
& \(r_m\) & \(r_n\) & \(r_{\mathrm{nnz}}\) \\
\midrule
Mittelmann & Gurobi
& 4.276 & 0.675 & 0.742 & 0.722 \\
& PSLP
& 1.112 & 0.813 & 0.824 & 0.798 \\
& GPU-Preprocessor
& \textbf{0.184} & \textbf{0.583}
& \textbf{0.580} & \textbf{0.572} \\
\midrule
MIPLIB Large & Gurobi
& 9.903 & 0.522 & 0.557 & 0.525 \\
& PSLP
& 3.214 & 0.600 & 0.635 & 0.608 \\
& GPU-Preprocessor
& \textbf{0.292} & \textbf{0.402}
& \textbf{0.384} & \textbf{0.342} \\
\midrule
QAP20 & Gurobi
& 183.654 & 0.019 & 0.500 & 0.405 \\
& PSLP
& 11.885 & 0.019 & 0.500 & 0.405 \\
& GPU-Preprocessor
& \textbf{0.271} & \textbf{0.014}
& \textbf{0.372} & \textbf{0.301} \\
\bottomrule
\end{tabular}
\end{table}

\subsection{End-to-End Solver Performance}
\label{sec:end-to-end}

To evaluate whether preprocessing reduces total solution time,
we compare HPR-LP without preprocessing with HPR-LP combined
with Gurobi Presolve, PSLP, or GPU-Preprocessor.

Table~\ref{tab:end-to-end-performance} shows that HPR-LP
with GPU-Preprocessor achieves the lowest total-time SGM10
on all three benchmark groups.
Relative to HPR-LP without preprocessing, it achieves
total-time speedups of \(2.5\times\)--\(5.0\times\).
The speedups over HPR-LP with PSLP and Gurobi Presolve
are \(1.5\times\)--\(3.0\times\) and
\(1.9\times\)--\(29.0\times\), respectively.
On Mittelmann, it solves all 49 instances, compared with
46--48 for the other configurations.
All four configurations solve 17 instances on MIPLIB Large
and 20 on QAP20. These results show that GPU-Preprocessor improves the
end-to-end performance of HPR-LP, as the reduction in
solve time exceeds the preprocessing cost.

\begin{table}[H]
\centering
\caption{End-to-end performance of HPR-LP with different
preprocessing methods. Bold values indicate the lowest
total-time SGM10 in each group.}
\label{tab:end-to-end-performance}
\vspace{4pt}
\small
\setlength{\tabcolsep}{5pt}
\resizebox{\linewidth}{!}{%
\begin{tabular}{llrrrr}
\toprule
Dataset & Method & Solved & Preprocessing time (s)
& Solve time (s) & Total time (s) \\
\midrule
Mittelmann & HPR-LP
& 46 & 0.000 & 14.456 & 14.456 \\
& HPR-LP + Gurobi
& 47 & 4.276 & 8.704 & 13.873 \\
& HPR-LP + PSLP
& 48 & 1.112 & 7.964 & 9.084 \\
& HPR-LP + GPU-Preprocessor
& 49 & 0.184 & 5.615 & \textbf{5.797} \\
\midrule
MIPLIB Large & HPR-LP
& 17 & 0.000 & 27.322 & 27.322 \\
& HPR-LP + Gurobi
& 17 & 9.903 & 10.689 & 20.534 \\
& HPR-LP + PSLP
& 17 & 3.214 & 12.038 & 15.721 \\
& HPR-LP + GPU-Preprocessor
& 17 & 0.292 & 10.407 & \textbf{10.666} \\
\midrule
QAP20 & HPR-LP
& 20 & 0.000 & 33.172 & 33.172 \\
& HPR-LP + Gurobi
& 20 & 183.654 & 8.080 & 194.165 \\
& HPR-LP + PSLP
& 20 & 11.885 & 7.934 & 20.220 \\
& HPR-LP + GPU-Preprocessor
& 20 & 0.271 & 6.407 & \textbf{6.687} \\
\bottomrule
\end{tabular}%
}
\end{table}

\subsection{Scalability on Extremely Large LPs}
\label{sec:gpu-advantage-at-scale}

We evaluate the scalability of GPU-Preprocessor using a family
of QAP LP relaxations generated from the QAPLIB instance
\texttt{wil100} with Hinder's large-scale LP test-problem
generator~\cite{hinder2024largescalelp}.
The underlying quadratic assignment problem is
\begin{equation}
    \min_{\pi\in\mathcal{S}_N}\quad
    \sum_{i=1}^{N}\sum_{k=1}^{N} F_{ik}D_{\pi(i),\pi(k)},
    \label{eq:qap-wil100-scale}
\end{equation}
where \(F\) is the flow matrix, \(D\) is the distance matrix,
and \(\mathcal{S}_N\) is the set of permutations of
\(N\) locations.
We construct instances for \(N\in\{20,40,60,80,100\}\)
by extracting the leading \(N\times N\) principal submatrices
of the flow and distance matrices of \texttt{wil100}.
All instances use the Adams--Johnson LP relaxation implemented
in the generator, with the LP formulation and column-scaling
rule fixed across problem sizes. At \(N=100\), the resulting LP has \(1{,}980{,}200\) rows,
\(49{,}015{,}000\) columns, and \(198{,}020{,}000\) nonzeros.

GPU-Preprocessor requires less than \(0.4\) seconds of
preprocessing at all five tested QAP sizes, as shown in
Figure~\ref{fig:qap-wil100-presolve-scale}.
PSLP is faster at \(N=20\), while GPU-Preprocessor achieves
lower preprocessing times than both baselines at the tested
sizes \(N\geq 40\).
At \(N=100\), it achieves the maximum preprocessing speedups
of \(126.3\times\) over PSLP and \(2{,}207.9\times\)
over Gurobi Presolve.

Table~\ref{tab:qap-wil100-reduction} shows that
GPU-Preprocessor also achieves lower row, column, and nonzero
ratios than both baselines at every tested size.
When combined with HPR-LP, it yields the lowest solve time
and total time across all five sizes.
At \(N=100\), the total-time speedups are \(16.0\times\)
over HPR-LP with PSLP and \(105.9\times\) over HPR-LP
with Gurobi Presolve.
These results show that the performance advantage of
GPU-Preprocessor becomes greater as problem size increases
within the tested QAP family.

\begin{figure}[H]
    \centering
    \includegraphics[width=0.92\linewidth]
    {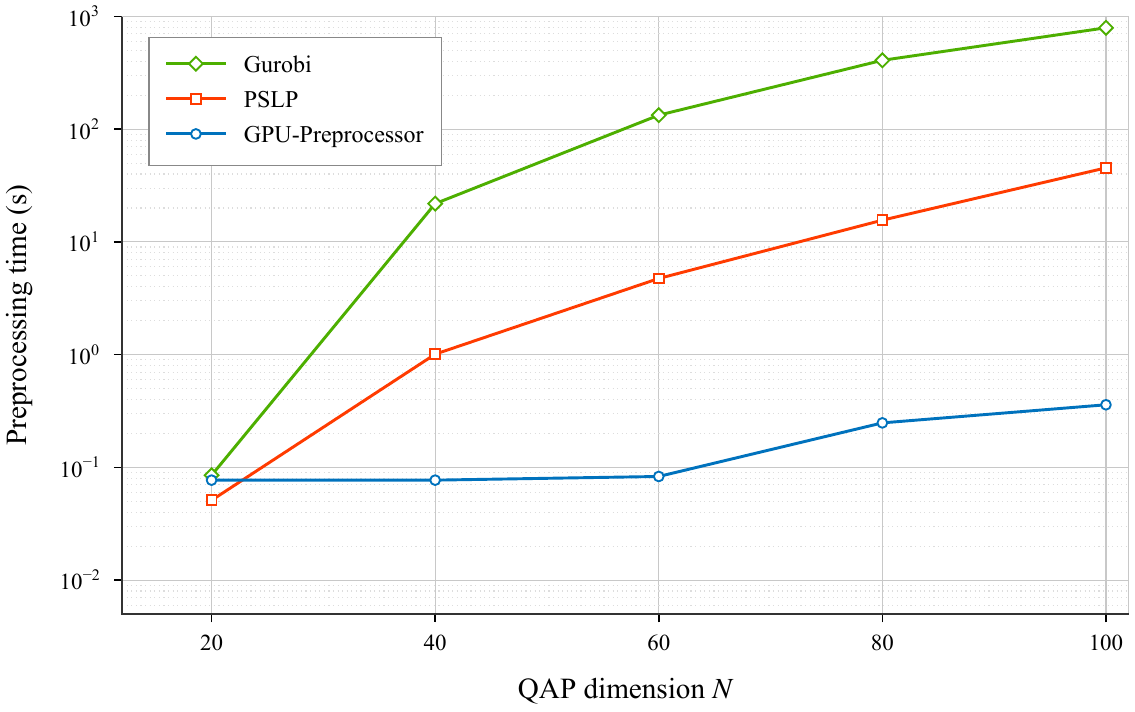}
    \caption{Preprocessing time versus QAP dimension \(N\).
    The vertical axis is logarithmic.}
    \label{fig:qap-wil100-presolve-scale}
\end{figure}

\begin{table}[H]
\centering
\caption{End-to-end performance and reduction strength for
individual QAP LP relaxations derived from \texttt{wil100}.
Bold values indicate the lowest value for each metric
at each \(N\).}
\label{tab:qap-wil100-reduction}
\vspace{4pt}
\small
\setlength{\tabcolsep}{3pt}
\begin{tabular}{rlrrrrrr}
\toprule
\(N\) & Preprocessor & Preprocessing time (s)
& Solve time (s) & Total time (s)
& \(r_m\) & \(r_n\) & \(r_{\mathrm{nnz}}\) \\
\midrule
20 & Gurobi
& 0.085 & 0.733 & 0.818 & 1.000 & 1.000 & 1.000 \\
& PSLP
& \textbf{0.051} & 0.779 & 0.829 & 1.000 & 1.000 & 1.000 \\
& GPU-Preprocessor
& 0.077 & \textbf{0.647} & \textbf{0.724}
& \textbf{0.239} & \textbf{0.226} & \textbf{0.227} \\
\midrule
40 & Gurobi
& 21.832 & 22.487 & 44.319 & 0.963 & 1.000 & 0.963 \\
& PSLP
& 1.009 & 3.864 & 4.873 & 1.000 & 1.000 & 1.000 \\
& GPU-Preprocessor
& \textbf{0.077} & \textbf{1.400} & \textbf{1.477}
& \textbf{0.239} & \textbf{0.228} & \textbf{0.228} \\
\midrule
60 & Gurobi
& 133.085 & 16.947 & 150.033 & 0.992 & 1.000 & 0.992 \\
& PSLP
& 4.735 & 11.962 & 16.698 & 1.000 & 1.000 & 1.000 \\
& GPU-Preprocessor
& \textbf{0.083} & \textbf{3.600} & \textbf{3.683}
& \textbf{0.240} & \textbf{0.230} & \textbf{0.230} \\
\midrule
80 & Gurobi
& 408.014 & 60.508 & 468.521 & 0.995 & 1.000 & 0.995 \\
& PSLP
& 15.580 & 46.268 & 61.848 & 1.000 & 1.000 & 1.000 \\
& GPU-Preprocessor
& \textbf{0.248} & \textbf{8.166} & \textbf{8.414}
& \textbf{0.241} & \textbf{0.232} & \textbf{0.232} \\
\midrule
100 & Gurobi
& 792.621 & 181.255 & 973.876 & 0.995 & 1.000 & 0.995 \\
& PSLP
& 45.357 & 101.869 & 147.227 & 1.000 & 1.000 & 1.000 \\
& GPU-Preprocessor
& \textbf{0.359} & \textbf{8.839} & \textbf{9.199}
& \textbf{0.122} & \textbf{0.094} & \textbf{0.094} \\
\bottomrule
\end{tabular}
\end{table}

\section{Conclusion}

This paper presented GPU-Preprocessor, a GPU-native LP preprocessing
framework that integrates folding and presolve.
To efficiently implement folding on GPUs, we adapted a randomized
partition refinement algorithm using sparse matrix--vector products
and parallel sorting.
For presolve, we introduced a new piecewise-linear reformulation
rule to obtain sparser LP representations and developed an efficient
GPU implementation that combines phase scheduling with a
Scan--Plan--Apply design to coordinate presolve rules and reduce
sparse-matrix reconstruction overhead through batched updates.
Numerical experiments demonstrate that GPU-Preprocessor achieves
effective model reductions and substantial preprocessing speedups,
leading to improved end-to-end LP solution performance.
An interesting direction for future research is to develop
GPU-native preprocessing techniques for convex quadratic
composite programming.

~\\

\noindent {\bf Acknowledgments.}
The work of Defeng Sun was supported by the Research Center for Intelligent Operations Research, RGC {Senior Research Fellow Scheme No. SRFS2223-5S02}, and  {GRF Project No. 15307523}. The work of Yancheng Yuan was supported by the NSFC Young Scientists Fund (Project No. 12501440) and the Research Center for Intelligent Operations Research. The work of Xinyuan Zhao was supported in part by the National Natural Science Foundation of China under Project {No. 12271015}.

\bibliographystyle{abbrv}
\bibliography{references}

\clearpage
\appendix

\section{Proof of Theorem~\ref{thm:folding-refinement-partition}}
\label{app:folding-refinement-proof}

\begin{proof}
First take $\theta=1$. Since updates only split parts, at most
$m+n-|\mathcal P_r^0|-|\mathcal P_c^0|$ rounds change a partition.
In column refinement, set
$a_j=(\sum_{i\in I}A_{ij})_{I\in\mathcal P_r^t}$.
Conditional on preceding samples, for $j,j'$ in the same current part
with $a_j\ne a_{j'}$,
\[
(h_c)_j-(h_c)_{j'}=\xi^\top(a_j-a_{j'})
\sim\mathcal N(0,\|a_j-a_{j'}\|_2^2),
\]
so their signatures differ with probability one; equal vectors give
equal signatures. The same argument applies to rows, conditional on
the updated column partition. Since only finitely many comparisons
occur, with probability one every refinement agrees with refinement
by the full coefficient-sum vectors. Work on this event below.

Initialization and refinement preserve condition~(i) of
Definition~\ref{def:folding-stable-revised-v2}, and a round with no
changes establishes condition~(ii). Any stable pair
$(\mathcal S_r,\mathcal S_c)$ refines the initial pair and, inductively,
every subsequent pair: each current row part is a union of parts of
$\mathcal S_r$, so columns in one part of $\mathcal S_c$ have equal
sums over it and remain together; the same reasoning applies to rows
after column refinement. Thus the returned stable pair is the unique
coarsest pair $(\mathcal P_r^\star,\mathcal P_c^\star)$.

Let $K$ count the rounds that change a partition. If row refinement
makes no change in round $t$, the pair is already stable.
For $t\geq2$, if column refinement makes no change, neither does row
refinement, since the preceding row step enforced the same equalities.
Hence, when $K\geq1$, row and column partitions must change in rounds
$1,\ldots,K-1$ and $2,\ldots,K$, respectively, giving
\[
K-1\leq\min\{m-|\mathcal P_r^0|,\,n-|\mathcal P_c^0|\}.
\]
Including the final unchanged round gives the claimed bound;
the case $K=0$ is immediate.

For general $\theta$, use the same samples until an early exit.
The part counts never exceed those of
$(\mathcal P_r^\star,\mathcal P_c^\star)$.
Thus, if $|\mathcal P_r^\star|\leq\theta m$ or
$|\mathcal P_c^\star|\leq\theta n$, the threshold test never succeeds
and the coarsest pair is returned. Otherwise, it succeeds no later
than the start of the final unchanged round, returning the trivial
stable pair. Early exit cannot increase the number of rounds.
\end{proof}

\section{Proof of Theorem~\ref{thm:unified-presolve-validity}}
\label{app:unified-presolve-proof}

In this appendix, a convex piecewise-linear function is a finite
maximum of affine functions:
\begin{equation}
f_i(\xi)=\max_{1\leq k\leq N_i}
\{\alpha_{ik}^{\top}\xi+\beta_{ik}\},
\qquad
\alpha_{ik}\in\mathbb R^d,\quad \beta_{ik}\in\mathbb R,
\label{eq:app-pwl-function}
\end{equation}
where $N_i$ is a positive integer. A convex P-LP has the form
\begin{equation}
\begin{aligned}
\min_{\xi\in\mathbb R^d}\quad & f_0(\xi)\\
\mathrm{s.t.}\quad & f_i(\xi)\leq0,\qquad i=1,\ldots,s,\\
& E\xi=h,
\end{aligned}
\label{eq:app-cpwl-problem}
\end{equation}
where $E\in\mathbb R^{r\times d}$, $h\in\mathbb R^r$, and each
$f_i$, including $f_0$, has the form~\eqref{eq:app-pwl-function}.

\subsection{From LPs to Convex P-LPs}

For an LP exhibiting either pattern, we construct a convex P-LP with the
same optimal value as the original LP and show how to recover an optimal
LP solution from an optimal solution of the convex P-LP.

\begin{proposition}[Pattern 1]
\label{prop:app-pattern1-recognition}
Suppose problem~\eqref{eq:original-lp} exhibits Pattern~1.
Let $\widehat L\leq\widehat A x\leq\widehat U$ denote the row
constraints remaining after removing $x_j\geq x_k$ and $x_j\geq-x_k$.
The following problem is a convex P-LP with the same feasible solutions
and objective function as problem~\eqref{eq:original-lp}:
\begin{equation}
\begin{aligned}
\min_x\quad & c^\top x\\
\mathrm{s.t.}\quad
& |x_k|-x_j\leq0,\\
& \widehat L\leq\widehat A x\leq\widehat U,\\
& \ell\leq x\leq u.
\end{aligned}
\label{eq:presolve-pattern1-plp}
\end{equation}
\end{proposition}

\begin{proof}
The inequalities $x_j\geq x_k$ and $x_j\geq-x_k$ are equivalent to
$|x_k|-x_j\leq0$. All other constraints and the objective function
$c^\top x$ are unchanged.
\end{proof}

\begin{proposition}[Pattern 2]
\label{prop:app-pattern2-recognition}
Suppose $p\ne q$ and the variables $x_p,x_q$ in problem~\eqref{eq:original-lp}
satisfy the conditions of Pattern~2. Define
\[
\alpha=\frac{c_p-c_q}{2},\quad
\beta=\frac{c_p+c_q}{2},\quad
d_i=\frac{A_{ip}-A_{iq}}{2},\quad
s_i=\frac{A_{ip}+A_{iq}}{2},\qquad i\in[m].
\]
The associated problem in the variables
$x_{[n]\setminus\{p,q\}}\in\mathbb R^{n-2}$ and $y\in\mathbb R$ is
\begin{equation}
\begin{array}{@{}>{\displaystyle}c@{\quad}>{\displaystyle}l@{\qquad}l@{}}
\min_{x_{[n]\setminus\{p,q\}},\,y}
& \sum_{j\in[n]\setminus\{p,q\}}c_jx_j+\alpha y+\beta|y|
& \\[3pt]
\mathrm{s.t.}
& L_i\leq\sum_{j\in[n]\setminus\{p,q\}}A_{ij}x_j+d_i y+s_i|y|\leq U_i,
& i\in[m],\\[3pt]
& \max\{y,0\}\leq u_p,\qquad \max\{-y,0\}\leq u_q,
& \\[3pt]
& \ell_j\leq x_j\leq u_j,
& j\in[n]\setminus\{p,q\}.
\end{array}
\label{eq:presolve-pattern2-underlying}
\end{equation}

\noindent Problem~\eqref{eq:presolve-pattern2-underlying} is a convex P-LP
with the following properties:
\begin{enumerate}[label=(\roman*),widest=ii,align=left,labelindent=0pt,leftmargin=*,nosep]
\setlength{\emergencystretch}{1em}
\item For any feasible solution $x$ of problem~\eqref{eq:original-lp},
$(x_{[n]\setminus\{p,q\}},x_p-x_q)$ is a feasible solution of
problem~\eqref{eq:presolve-pattern2-underlying} with an objective value
no greater than the objective value of $x$ in problem~\eqref{eq:original-lp}.
\item For any feasible solution $(x_{[n]\setminus\{p,q\}},y)$ of
problem~\eqref{eq:presolve-pattern2-underlying}, set $x_p=\max\{y,0\}$
and $x_q=\max\{-y,0\}$, keeping $x_j$ unchanged for all
$j\in[n]\setminus\{p,q\}$. Then $x$ is a feasible solution of
problem~\eqref{eq:original-lp} with the same objective value as
$(x_{[n]\setminus\{p,q\}},y)$ in
problem~\eqref{eq:presolve-pattern2-underlying}.
\end{enumerate}
\end{proposition}

\begin{proof}
Since $\beta\geq0$, $s_i\geq0$ for all $i\in[m]$ with $U_i<+\infty$,
and $-s_i\geq0$ for all $i\in[m]$ with $L_i>-\infty$,
problem~\eqref{eq:presolve-pattern2-underlying} is a convex P-LP.

\medskip\noindent\emph{Proof of (i).}
For any feasible solution $x$ of problem~\eqref{eq:original-lp},
set $y=x_p-x_q$ and retain the other coordinates.
Let $\delta=\min\{x_p,x_q\}$. Since $x_p,x_q\geq0$, we have
$\delta\geq0$ and $|y|=|x_p-x_q|=x_p+x_q-2\delta$. Hence
\[
\begin{aligned}
d_i y+s_i|y|&=(d_i+s_i)x_p+(s_i-d_i)x_q-2s_i\delta
=A_{ip}x_p+A_{iq}x_q-2s_i\delta,\qquad i\in[m],\\[2pt]
\alpha y+\beta|y|&=(\alpha+\beta)x_p+(\beta-\alpha)x_q-2\beta\delta
=c_px_p+c_qx_q-2\beta\delta.
\end{aligned}
\]
Since $\beta\geq0$, $s_i\geq0$ for all $i\in[m]$ with $U_i<+\infty$,
$s_i\leq0$ for all $i\in[m]$ with $L_i>-\infty$,
$\max\{y,0\}\leq x_p$, and $\max\{-y,0\}\leq x_q$,
$(x_{[n]\setminus\{p,q\}},y)$ is a feasible solution of
problem~\eqref{eq:presolve-pattern2-underlying}
and has an objective value no greater than the objective value of $x$
in problem~\eqref{eq:original-lp}.

\medskip\noindent\emph{Proof of (ii).}
Let $(x_{[n]\setminus\{p,q\}},y)$ be a feasible solution of
problem~\eqref{eq:presolve-pattern2-underlying}.
Set $x_p=\max\{y,0\}$ and $x_q=\max\{-y,0\}$, keeping $x_j$
unchanged for all $j\in[n]\setminus\{p,q\}$.
Then
\[
\begin{aligned}
x_p-x_q&=\max\{y,0\}-\max\{-y,0\}=y,\\
x_p+x_q&=\max\{y,0\}+\max\{-y,0\}=|y|.
\end{aligned}
\]
We obtain
\[
\begin{aligned}
(Ax)_i&=\sum_{j\in[n]\setminus\{p,q\}}A_{ij}x_j
+A_{ip}x_p+A_{iq}x_q
=\sum_{j\in[n]\setminus\{p,q\}}A_{ij}x_j
+d_i(x_p-x_q)+s_i(x_p+x_q)\\
&=\sum_{j\in[n]\setminus\{p,q\}}A_{ij}x_j
+d_i y+s_i|y|,\qquad i\in[m],\\[3pt]
c^\top x&=\sum_{j\in[n]\setminus\{p,q\}}c_jx_j
+c_px_p+c_qx_q
=\sum_{j\in[n]\setminus\{p,q\}}c_jx_j
+\alpha(x_p-x_q)+\beta(x_p+x_q)\\
&=\sum_{j\in[n]\setminus\{p,q\}}c_jx_j
+\alpha y+\beta|y|.
\end{aligned}
\]
The constraints of problem~\eqref{eq:presolve-pattern2-underlying}
give $L\leq Ax\leq U$, $x_p\leq u_p$, and $x_q\leq u_q$.
Since $x_p,x_q\geq0$, $\ell_p=\ell_q=0$, and the other coordinates
retain their bounds, we also have $\ell\leq x\leq u$.
Thus $x$ is a feasible solution of problem~\eqref{eq:original-lp}
with the same objective value as $(x_{[n]\setminus\{p,q\}},y)$ in
problem~\eqref{eq:presolve-pattern2-underlying}.
\end{proof}

\subsection{From Convex P-LPs to LPs}

The following proposition gives two LP representations of
problem~\eqref{eq:app-cpwl-problem}, one exhibiting each pattern.

\begin{proposition}
\label{prop:app-universal-patterns}
For the functions in~\eqref{eq:app-pwl-function}, write
\[
g_{ik}(\xi)=\alpha_{ik}^{\top}\xi+\beta_{ik},
\qquad i=0,\ldots,s,\quad k\in[N_i].
\]
We may assume $N_i\geq2$ for every $i=0,\ldots,s$ by duplicating
an affine piece when necessary.
Let $\xi\in\mathbb R^d$ and $t\in\mathbb R$.
The vectors $y,z,p,q$ consist of real auxiliary variables
$y_{ik},z_{ik},p_{ik},q_{ik}$, respectively, indexed by
$i=0,\ldots,s$ and $k=2,\ldots,N_i$.
The following LPs represent problem~\eqref{eq:app-cpwl-problem}
and exhibit Patterns~1 and~2, respectively:
\begin{equation}
\begin{array}{@{}>{\displaystyle}c@{\quad}>{\displaystyle}l@{\qquad}l@{}}
\min_{\xi,t,y,z} & t & \\[3pt]
\mathrm{s.t.}
& y_{ik}=g_{ik}(\xi)-g_{i1}(\xi),
& i=0,\ldots,s,\ k=2,\ldots,N_i,\\[3pt]
& z_{ik}\geq y_{ik},\qquad z_{ik}\geq-y_{ik},
& i=0,\ldots,s,\ k=2,\ldots,N_i,\\[3pt]
& g_{0k}(\xi)+g_{01}(\xi)+z_{0k}\leq2t,
& k=2,\ldots,N_0,\\[3pt]
& g_{ik}(\xi)+g_{i1}(\xi)+z_{ik}\leq0,
& i\in[s],\ k=2,\ldots,N_i,\\[3pt]
& E\xi=h.
\end{array}
\label{eq:app-universal-pattern1}
\end{equation}
\begin{equation}
\begin{array}{@{}>{\displaystyle}c@{\quad}>{\displaystyle}l@{\qquad}l@{}}
\min_{\xi,t,p,q} & t & \\[3pt]
\mathrm{s.t.}
& p_{ik}-q_{ik}=g_{ik}(\xi)-g_{i1}(\xi),
& i=0,\ldots,s,\ k=2,\ldots,N_i,\\[3pt]
& p_{ik},q_{ik}\geq0,
& i=0,\ldots,s,\ k=2,\ldots,N_i,\\[3pt]
& g_{0k}(\xi)+g_{01}(\xi)+p_{0k}+q_{0k}\leq2t,
& k=2,\ldots,N_0,\\[3pt]
& g_{ik}(\xi)+g_{i1}(\xi)+p_{ik}+q_{ik}\leq0,
& i\in[s],\ k=2,\ldots,N_i,\\[3pt]
& E\xi=h.
\end{array}
\label{eq:app-universal-pattern2}
\end{equation}
Both LPs have the same optimal value as
problem~\eqref{eq:app-cpwl-problem}.
If $(\xi^\star,t^\star,y^\star,z^\star)$ solves
problem~\eqref{eq:app-universal-pattern1}, or
$(\xi^\star,t^\star,p^\star,q^\star)$ solves
problem~\eqref{eq:app-universal-pattern2}, then $\xi^\star$ solves
problem~\eqref{eq:app-cpwl-problem}.
For every optimal solution $\xi^\star$ of
problem~\eqref{eq:app-cpwl-problem}, each LP has an optimal solution
with $\xi=\xi^\star$.
\end{proposition}

\begin{proof}
The inequalities $z_{ik}\geq y_{ik}$ and $z_{ik}\geq-y_{ik}$ in
problem~\eqref{eq:app-universal-pattern1} constitute Pattern~1.
In problem~\eqref{eq:app-universal-pattern2}, each pair
$(p_{ik},q_{ik})$ with $i=0,\ldots,s$ and $k=2,\ldots,N_i$
satisfies the conditions of Pattern~2.

We first show that every feasible solution of
problem~\eqref{eq:app-universal-pattern1} or~\eqref{eq:app-universal-pattern2} gives a
feasible solution $\xi$ of problem~\eqref{eq:app-cpwl-problem}, with
$f_0(\xi)\leq t$.
Let $(\xi,t,y,z)$ be a feasible solution of
problem~\eqref{eq:app-universal-pattern1}.
Its constraints give
$z_{ik}\geq|y_{ik}|=|g_{ik}(\xi)-g_{i1}(\xi)|$ for
$i=0,\ldots,s$ and $k=2,\ldots,N_i$.
For $i\in[s]$ and $k=2,\ldots,N_i$, we therefore have
\[
\begin{aligned}
2\max\{g_{i1}(\xi),g_{ik}(\xi)\}
&=g_{i1}(\xi)+g_{ik}(\xi)+|g_{ik}(\xi)-g_{i1}(\xi)|\\
&\leq g_{i1}(\xi)+g_{ik}(\xi)+z_{ik}\leq0.
\end{aligned}
\]
Thus $g_{i1}(\xi)\leq0$ and $g_{ik}(\xi)\leq0$ for every
$k=2,\ldots,N_i$, giving
$f_i(\xi)=\max_{k\in[N_i]}g_{ik}(\xi)\leq0$ for all $i\in[s]$.
Similarly, for $k=2,\ldots,N_0$,
\[
2\max\{g_{01}(\xi),g_{0k}(\xi)\}
\leq g_{01}(\xi)+g_{0k}(\xi)+z_{0k}\leq2t,
\]
so $f_0(\xi)=\max_{k\in[N_0]}g_{0k}(\xi)\leq t$.
Since $E\xi=h$ also holds, $\xi$ is feasible for
problem~\eqref{eq:app-cpwl-problem}, with an objective value
$f_0(\xi)$ no greater than the objective value $t$ of
$(\xi,t,y,z)$ in problem~\eqref{eq:app-universal-pattern1}.

For a feasible solution $(\xi,t,p,q)$ of
problem~\eqref{eq:app-universal-pattern2}, the constraints give
\[
p_{ik}+q_{ik}\geq|p_{ik}-q_{ik}|
=|g_{ik}(\xi)-g_{i1}(\xi)|,
\qquad i=0,\ldots,s,\quad k=2,\ldots,N_i.
\]
Replacing $z_{ik}$ by $p_{ik}+q_{ik}$ in the preceding inequalities
gives $f_i(\xi)\leq0$ for all $i\in[s]$ and $f_0(\xi)\leq t$.
Together with the constraint $E\xi=h$, this shows that $\xi$ is
feasible for problem~\eqref{eq:app-cpwl-problem}, with an objective
value $f_0(\xi)$ no greater than the objective value $t$ of
$(\xi,t,p,q)$ in problem~\eqref{eq:app-universal-pattern2}.

We next show that every feasible solution $\xi$ of
problem~\eqref{eq:app-cpwl-problem} gives feasible solutions of
problems~\eqref{eq:app-universal-pattern1} and~\eqref{eq:app-universal-pattern2},
each with objective value $f_0(\xi)$.
Set $t=f_0(\xi)$.
For $i=0,\ldots,s$ and $k=2,\ldots,N_i$, choose
\[
y_{ik}=g_{ik}(\xi)-g_{i1}(\xi),\qquad
z_{ik}=|g_{ik}(\xi)-g_{i1}(\xi)|
\]
in problem~\eqref{eq:app-universal-pattern1}, and
\[
p_{ik}=\max\{g_{ik}(\xi)-g_{i1}(\xi),0\},\qquad
q_{ik}=\max\{g_{i1}(\xi)-g_{ik}(\xi),0\}
\]
in problem~\eqref{eq:app-universal-pattern2}.
Both resulting LP solutions are feasible and have objective value
$t=f_0(\xi)$, equal to the objective value of $\xi$ in
problem~\eqref{eq:app-cpwl-problem}.

It follows that both LPs have the same optimal value as
problem~\eqref{eq:app-cpwl-problem}.
If $\xi^\star$ is optimal for problem~\eqref{eq:app-cpwl-problem},
the choices above give optimal solutions of both LPs.
If $(\xi^\star,t^\star,y^\star,z^\star)$ is an optimal solution of
problem~\eqref{eq:app-universal-pattern1}, or $(\xi^\star,t^\star,p^\star,q^\star)$ is an
optimal solution of problem~\eqref{eq:app-universal-pattern2}, then
$\xi^\star$ is feasible for problem~\eqref{eq:app-cpwl-problem} and
$f_0(\xi^\star)\leq t^\star$.
Since $t^\star$ equals the optimal value of
problem~\eqref{eq:app-cpwl-problem}, $\xi^\star$ is an optimal solution of
problem~\eqref{eq:app-cpwl-problem}.
\end{proof}

\subsection{Validity of the Piecewise-Linear Reformulation Rule}
\label{app:patterns-rule-validity}

\begin{proof}[Proof of Theorem~\ref{thm:unified-presolve-validity}]
Suppose the original LP has an optimal solution.
By Propositions~\ref{prop:app-pattern1-recognition}
and~\ref{prop:app-pattern2-recognition}, the recognized
convex P-LP has an optimal solution and the same optimal value
as the original LP.
The valid presolve step then produces a presolved convex P-LP
with an optimal solution and the same optimal value.
By Proposition~\ref{prop:app-universal-patterns}, the resulting
presolved LP also has an optimal solution and the same optimal
value as the original LP.

Every optimal solution of the presolved LP gives an optimal
solution of the presolved convex P-LP by
Proposition~\ref{prop:app-universal-patterns}.
The presolve recovery map then recovers an optimal solution of the
recognized convex P-LP, from which the constructions in
Propositions~\ref{prop:app-pattern1-recognition}
and~\ref{prop:app-pattern2-recognition} recover an optimal solution
of the original LP.
Composing these maps gives the recovery map required by
Definition~\ref{def:valid-preprocessing-operator}.
\end{proof}

\end{document}